\documentclass[reqno]{amsart}
\usepackage{diagrams}
\usepackage{amssymb}
\usepackage{mathrsfs}
\usepackage{cite}
\usepackage{tikz-cd}
\usepackage{MnSymbol}
\usepackage{a4wide}
 
\usepackage{xcolor}

\DeclareMathOperator\Z{\mathbb Z}
\newcommand{\Om}{\Omega}

\newcommand{\G}{{\mathbb G}}

\newtheorem{theorem}{Theorem}[section]
\newtheorem{lemma}[theorem]{Lemma}
\newtheorem{cor}[theorem]{Corollary}
\newtheorem{prop}[theorem]{Proposition}
\theoremstyle{definition}

\newtheorem{example}[theorem]{Example}

\theoremstyle{remark}
\newtheorem{remark}[theorem]{Remark}

\newcommand{\dontprint}[1]\relax

\newcommand{\Th}{\Theta}

\renewcommand{\P}{{\mathbb P}}
\newcommand{\A}{{\mathbb A}}
\newcommand{\wt}{\widetilde}
\newcommand{\ot}{\otimes}

\newcommand{\Hom}{\operatorname{Hom}}
\newcommand{\End}{\operatorname{End}}

\newcommand{\Ext}{\operatorname{Ext}}
\renewcommand{\AA}{{\mathcal A}}
\newcommand{\BB}{{\mathcal B}}

\newcommand{\NN}{{\mathcal N}}

\newcommand{\OO}{{\mathcal O}}

\newcommand{\si}{\sigma}
\newcommand{\de}{\delta}
\newcommand{\sub}{\subset}

\newcommand{\Spec}{\operatorname{Spec}}

\newcommand{\Nm}{\operatorname{Nm}}

\newcommand{\lan}{\langle}
\newcommand{\ran}{\rangle}
\newcommand{\ov}{\overline}
\newcommand{\im}{\operatorname{im}}
\newcommand{\om}{\omega}

\renewcommand{\a}{\alpha}

\newcommand{\tr}{\operatorname{tr}}
\newcommand{\id}{\operatorname{id}}
\newcommand{\und}{\underline}
\renewcommand{\th}{\theta}
\newcommand{\coker}{\operatorname{coker}}

\newcommand{\hra}{\hookrightarrow}
\newcommand{\we}{\wedge}

\newcommand{\rk}{\operatorname{rk}}

\newcommand{\eps}{\epsilon}

\newcommand{\gr}{\operatorname{gr}}

\numberwithin{equation}{section}
\title{De Rham cohomology for supervarieties}
\author{Alexander Polishchuk}
\thanks{Partially supported by the NSF grant DMS-2001224, 
and within the framework of the HSE University Basic Research Program and by the Russian Academic Excellence Project `5-100'.}

\address{
    Department of Mathematics, 
    University of Oregon, 
    Eugene, OR 97403, USA; and National Research University Higher School of Economics, Moscow, Russia
  }
  \email{apolish@uoregon.edu}

\begin{document}

\begin{abstract} We study the de Rham cohomology and the Hodge to de Rham spectral sequence for supervarieties.
\end{abstract}

\maketitle

\section{Basics}

\subsection{Introduction}

Recall (see e.g.,\cite{Gr}) that for a smooth proper scheme $X_0$ over a field $k$, one defines the de Rham cohomology
$H_{dR}(X_0)$ as the hypercohomology of the de Rham complex $\Om^\bullet_{X_0}$.
One has a natural {\it Hodge to de Rham} spectral sequence
$$E_1^{pq}(X_0)=H^q(X,\Om^p_{X_0}) \implies H^{p+q}_{dR}(X_0).$$
In characteristic zero it is known to degenerate at $E_1$ (this can be deduced from the Hodge theory).

In this paper we study the similar spectral sequence for a smooth proper supervariety $X$.
It is easy to see that it often does not degenerate at $E_1$, and its behavior is much more complicated.

We prove that the de Rham cohomology of a smooth supervariety $X$ coincides with the de Rham cohomology of the underlying variety $X_0$ (see Theorem \ref{deR-thm}).
We show that the Hodge to de Rham spectral sequence of a smooth proper supervariety $X$ degenerates at $E_2$ if $X$ is split (see Theorem \ref{split-thm}) or has dimension $n|2$ (see 
Theorem \ref{n2-dim-thm}). In the latter case we also give a formula for dimensions of the terms $E_2^{pq}(X)$.
On the other hand, we give an example showing that this spectral sequence does not always degenerate at the $E_2$ page if $X$ has dimension $1|3$ (see Theorem \ref{nondeg-E2-thm}).
This answers negatively the question raised in \cite[Problem 5.14]{CNR}.

An interesting phenomenon is that the limiting terms $(E_\infty^{pq}(X))$ of the Hodge to de Rham sequence can differ from $(E_\infty^{pq}(X_0))$ 
%although the sums of dimensions over $p+q=i$ are the same
(e.g., this happens for the supergrassmannian $G(1|1,2|2)$, see Example \ref{Gr-ex}).
Thus, the way the dimension of $H^i_{dR}(X_0)$ gets distributed among $E_\infty^{pq}(X)$ with $p+q=i$ gives interesting numerical invariants of a supervariety
(which cannot be recovered from $X_0$ unless $X$ is projected).

We refer to \cite{Manin} and \cite{BHRP} for basics on algebraic supergeometry. We work with algebraic supervarieties and Zariski topology, but the analogous picture holds
in the complex analytic case. Note that even though one can introduce analogs of $(p,q)$-forms on the CS manifold underlying a complex supermanifold,
there seems to be no analog of the Hodge decomposition in the presence of odd variables. The obvious problem is that the de Rham complex $\Om^\bullet_X$ is unbounded,
and the twisted dual $(\Om^\bullet_X)^\vee\ot \om_X$ (where $\om_X$ is the Berezinian) is a different complex, the complex of integral forms.

In many respects the study of supervarieties is similar to the study of nilpotent extensions of smooth schemes (in the category of usual even schemes).
As a manifestation of this, we prove a result in even algebraic geometry (inspired by the analogy with supervarieties) that states surjectivity of the map on cohomology of structure sheaves
for nilpotent extensions of smooth proper schemes in characteristic zero  (see Theorem \ref{nilp-sur-thm}).

The paper is organized as follows. 
In Sec.\ \ref{deR-sec} we prove the comparison theorem on de Rham cohomology of a supervariety $X$ (Theorem \ref{deR-thm}). In Sec.\ \ref{deR-rel-sec} we
discuss relative de Rham cohomology and the Gauss-Manin connection. In Sec.\ \ref{H-dR-sec} we present some general observations on the Hodge to de Rham
spectral sequence for supervarieties. In Sec.\ \ref{even-sec} we prove a result on even nilpotent extensions of smooth proper schemes (Theorem \ref{nilp-sur-thm}).
Next, in Sec.\ \ref{comp-sec} we present some more specific computations. In Sec.\ \ref{split-sec} we consider split supervarieties as well as supervarieties in characteristic $2$, 
in Sec.\ \ref{n2-sec}---
supervarieties of dimension $n|2$, and in Sec. \ref{13-dim-sec}---an example showing that in dimension $1|3$, the differential $d:(\OO_X)_-\to (\Om^1_X)_-$
does not always induce injection on $H^1$.

\bigskip

\noindent
{\it Acknowledgments}. I am grateful to Vera Serganova and to Johan de Jong for useful discussions. Part of this work was done while the author visited 
the Simons Center for Geometry and Physics. I'd like to thank this institution for hospitality and excellent working conditions.

\bigskip

\noindent
{\it Conventions}. For a $\Z/2$-graded object $X$ we denote by $X_+$ and $X_-$ its even and odd parts.
We use the parity of differential forms such that the de Rham differential is even, so the de Rham complex can be viewed as a complex of $\Z/2$-graded sheaves.

\subsection{De Rham cohomology of a supervariety}\label{deR-sec}

%Let $\pi:X\to S$ be a smooth superscheme, proper over a base superscheme $S$.
Let $X$ be a smooth superscheme over a field $k$.
Then we have the de Rham complex $\Om_X^\bullet=(\Om_X^\bullet,d)$, which we view as a bounded below complex of $k$-modules in Zariski topology.
Therefore, we can define de Rham cohomology of $X$ as its hypercohomology:
$$H^\bullet_{dR}(X):=H^\bullet(X,\Om_X^\bullet).$$
Note that as soon as $X$ has nonzero odd dimension, the complex $\Om_X^\bullet$ is not bounded.
%relative de Rham complex $(\Om_{X/S}^\bullet,d)$, which we view as a (bounded below) complex of $\pi^{-1}\OO_S$-modules in Zariski topology
%Hence, we can consider $R\pi_*(\Om_{X/S}^\bullet)$ as an object in the (bounded below) derived category of $\OO_S$-modules.
%The case when $S$ is a point. 

Let $\NN\sub \OO_X$ denote the ideal generated by odd functions,
and let $X_0$ denote the bosonization of $X$, so $\OO_{X_0}=\OO_X/\NN$. 
We have a natural surjective map of complexes of $k$-vector spaces
$$\Om_X^\bullet\to \Om_{X_0}^\bullet.$$
Let us denote by $K$ its kernel.

\begin{theorem}\label{deR-thm} 
The complex $K$ is acyclic. Hence, one has a natural isomorphism 
$H_{dR}(X)\rTo{\sim} H_{dR}(X_0)$.
\end{theorem}

\begin{proof}
Let us introduce the decreasing algebra filtration $(F^\bullet_\NN)$ on $\Om_X^\bullet$ by letting
$F^i_\NN$ be the ideal (with respect to the algebra structure on $\Om_X^\bullet$) generated by
$$(\NN^i,\NN^{i-1}d\NN,\NN^{i-2}(d\NN)^2,\ldots,(d\NN)^i).$$
Then $K^\bullet=F^1_\NN$.

We claim that for $i\ge 1$, the complex $\gr^i_F:=F^i_\NN/F^{i+1}_\NN$ is exact.
Let us consider the $\OO_{X_0}$-algebra $\AA:=\gr^\bullet_F\Om^\bullet$, equipped with the differential $d_\AA$ induced by $d$.
Note that $\gr^0_F=\Om^\bullet_{X_0}$. We have an embedding of $\OO_{X_0}$-algebras 
$$\AA_0:=\Om^\bullet_{X_0}\ot {\bigwedge}^\bullet(\NN/\NN^2)\hra \AA.$$
Let us set 
$$\AA_n:=\AA_0\cdot \gr^{\le n}_F\sub \AA.$$
Then $(\AA_n)$ is an increasing algebra filtration, and we have an algebra isomorphism
$$\ov{\AA}:=\bigoplus_n \AA_n/\AA_{n-1}\simeq \AA_0\ot S^\bullet(\NN/\NN^2)\simeq \Om^\bullet_{X_0}\ot K(\NN/\NN^2),$$
where $K(\NN/\NN^2):={\bigwedge}^\bullet(\NN/\NN^2)\ot S^\bullet(\NN/\NN^2)$.
Furthermore, the differential $d_\AA$ sends $\AA_n$ to $\AA_{n+1}$, and the induced differential $\AA_n/\AA_{n-1}\to \AA_{n+1}/\AA_n$
on $\ov{\AA}$ is induced by the Koszul differential on $K(\NN/\NN^2)$. Finally, we observe that 
$\bigoplus_n (\AA_n\cap \gr^{\ge 1}_F)/(\AA_{n-1}\cap \gr^{\ge 1}_F)$ gets identified with $\Om^\bullet_{X_0}\ot K_{\ge 1}(\NN/\NN^2)$,
where $K_{\ge 1}(\NN/\NN^2)\sub K(\NN/\NN^2)$ is the subcomplex of elements of weight $\ge 1$ (i.e., the augmentation ideal). Now
our assertion follows from the exactness of $K_{\ge 1}(\NN/\NN^2)$ (in fact, it is homotopic to zero).
%We also have a second grading on $\AA$ induced by the grading on $\Om^\bullet$, 
%where $G^j\gr^i_F\Om^p$ is the image of $(\NN^{j-p})\cap F^i\Om^p$.
\end{proof}

\begin{remark}
Note that Theorem \ref{deR-thm} holds in any characteristic.
In the complex analytic context (and hence, in characteristic zero) Theorem \ref{deR-thm} can also be deduced from the fact that $\Om_X^\bullet$ is a resolution of the constant sheaf on $X$.
\end{remark}

\subsection{Relative de Rham cohomology, actions of supergroups and the Gauss-Manin connection}\label{deR-rel-sec}

More generally, if $f:X\to S$ is a smooth proper morphism of superschemes then we have the relative de Rham complex $\Om_{X/S}^\bullet$,
and we can form the corresponding relative de Rham cohomology sheaves of $\OO$-modules on $S$,
$$H^i_{dR}(X/S):=\und{H}^i Rf_*(\Om^\bullet_{X/S}).$$
Assume that $S$ is Noetherian. Since the sheaves $R^q(\Om^p_{X/S})$ are coherent, it follows that $H^i_{dR}(X/S)$ are coherent sheaves.

It is easy to see that the formation of $H^i_{dR}(X/S)$ is compatible with flat base changes $S'\to S$.
In particular, if we start with $X$, proper and smooth over the ground field $k$, then for any Noetherian commutative $k$-superalgebra $A$, we can consider the base
change $X_A\to \Spec(A)$, and we have natural isomorphisms
$$H^i_{dR}(X)\ot A\simeq H^i_{dR}(X_A/A).$$

This shows that if an algebraic $k$-supergroup $G$ acts on $X$ then $H^i_{dR}(X)$ have natural structure of $G$-representations.
Indeed, for $A$ as above, the group $G(A)$ acts on $X_A$ and this gives its action on $H^i_{dR}(X_A/A)\simeq H^i_{dR}(X)\ot A$. 

\begin{prop} Assume the characteristic is zero. 
If the underlying usual algebraic group $G_0$ is connected then $H^i_{dR}(X)$ is trivial as a $G$-representation.
\end{prop}

\begin{proof}
It suffices to show that the corresponding Lie superalgebra acts trivially. Thus, it is enough to show that for every global vector field $v$ (even or odd),
the automorphism $f\mapsto f+v(f)\eps$ of $X_{D}$, where $D=k[\eps]/(\eps^2)$ ($\eps$ is either even or odd), induces the trivial automorphism
of $H^i_{dR}(X_D/D)$. But this immediately follows from the fact that the corresponding automorphism of the relative de Rham complex $\Om^\bullet_{X_D/D}$,
given by $\eta\mapsto \eta+L_v(\eta)\eps$ (where $L_v$ is the Lie derivative)
is homotopic to the identity via the homotopy $\eta\mapsto i_v(\eta)$ (the contraction by $v$).
\end{proof}

In the case when $f:X\to S$ is a smooth proper morphism, where $S$ is smooth over a field $k$, similarly to the classical case (see e.g., \cite[Sec.\ IV]{Katz}) one can
construct the {\it Gauss-Manin connection} on $H^i_{dR}(X/S)$. Namely, consider the filtration $F^i$ on the absolute de Rham complex $\Om^\bullet_X$, where
$F^i$ is the image of $f^*\Om^i_S\ot \Om^{\bullet-i}_X\to \Om^\bullet_X$. Then we have natural identifications
$$\gr^i_F \Om^\bullet_X\simeq f^*\Om^i_S\ot \Om^{\bullet}_{X/S}[-i].$$
The connection on $H^i_{dR}(X/S)$ is defined as the connecting homomorphism
$$\nabla:H^i_{dR}(X/S)\simeq R^if_*\gr^0_F\to R^{i+1}f_*\gr^1_F\simeq \Om^1_S\ot H^i_{dR}(X/S).$$ 
Its integrability is proved by considering the entire page of the corresponding spectral sequence.

As in the classical case, this implies that for smooth $S$, the coherent sheaves $H^i_{dR}(X/S)$ on $S$ are in fact locally free (using \cite[Sec.\ 2.2.1]{Penkov}),
and that the formation of the $H^i_{dR}(X/S)$ is compatible with arbitrary base change. 

Note that 
the sheaves $R^qf_*\Om^p_X$ are not necessarily locally free (they are in the even case, in characteristic zero): for example, 
there exist families of supercurves $f:X\to S$ such that $R^1f_*\OO_X$
is not locally free (see \cite[Sec.\ 3.3]{FKP}).

%On the other hand, similarly to the classical case, for a smooth proper morphism of superschemes $f:X\to S$ one has ???

\subsection{Spectral sequence: general observations}\label{H-dR-sec}

From now on we assume that $X$ is proper and smooth over the ground field $k$.

As in the even case, we associate with a complex of sheaves $\Om_X^\bullet$ the 
%first quadrant 
spectral sequence
$$E_1^{pq}(X)=H^q(X,\Om^p_X) \implies H^{p+q}_{dR}(X),$$
to which we refer as {\it Hodge to de Rham} spectral sequence.
By Theorem \ref{deR-thm}, its limit $H^\bullet_{dR}(X)$ is identified with $H^\bullet_{dR}(X_0)$

In the even case, assuming that the characteristic is zero, the Hodge to de Rham spectral sequence
is known to degenerate at the page $E_1$. There is not much we can say about this spectral sequence for general supervarieties, even in characteristic zero.
Already the example of $X$ of dimension $0|n$ shows that it does not have to degenerate at $E_1$. 
We will prove that it degenerates at $E_2$ if $X$ is split or has dimension $n|2$ (see Theorems \ref{split-thm} and \ref{n2-dim-thm}).
We will also show that there exist $X$ of dimension $1|3$ such that the Hodge to de Rham spectral sequence of $X$ does not degenerate at $E_2$
(see Theorem \ref{nondeg-E2-thm}).

At present, we do not know an example of $X$ in characteristic zero such that $E_2(X)$ is not finite-dimensional.

%It is natural to pose the following questions. For both questions we assume that the characteristic is zero. 

%\medskip
%{\bf Question}. {\it Is $E_2(X)$ finite-dimensional?}

%\medskip
%{\bf Question 2}. {\it Does the Hodge to de Rham spectral sequence of $X$ degenerate at $E_2$?}

\medskip

%Note that a positive answer to Question 2 implies a positive answer to Question 1. We will show below that the answer to both Questions is positive if

\begin{prop}
We always have an isomorphism $E_2^{00}(X)\simeq H^0_{dR}(X_0)$
%\simeq =\ker(H^0(X,\OO)\to H^0(X,\Om^1))$ 
and an injection
\begin{equation}\label{E2-10-inj}
E_2^{10}(X)\hra H^0(\Om^1_{X_0}).
\end{equation}
\end{prop}

\begin{proof}
The first assertion is clear since the spectral sequence degenerates at $E_2^{00}$. 
%Next, we have an exact sequence
%\begin{equation}
%0\to E_2^{10}(X)\to H^1_{dR}(X_0)\to E_2^{01}(X)\to E_2^{20}(X).
%\end{equation}
It also degenerates at $E_2^{10}$, so we have an embedding $E_2^{10}(X)\to H^1_{dR}(X)\simeq H^1_{dR}(X_0)$.
This map factors as the composition
$$E_2^{10}(X)\to H^0(\Om^1_{X_0})\to H^1_{dR}(X_0),$$
which implies the claimed injection.
\end{proof}

For any morphism $f:X\to Y$ between smooth proper varieties, the pull-back gives a morphism of spectral sequences
$$f^*:E_n^{pq}(Y)\to E_n^{pq}(X).$$
Applying this to the embedding $i:X_0\to X$, we get a morphism of spectral sequences
$$i^*:E_n^{pq}(X)\to E_n^{pq}(X_0)$$
A curious fact is that although the morphism $H_{dR}(X)\to H_{dR}(X_0)$ is always an isomorphism, if $X$ is not projected,
the morphism $E_\infty^{pq}(X)\to E_{\infty}^{pq}(X_0)$ is not necessarily an isomorphism (for example, this happens for $X=G(1|1,2|2)$, see Example \ref{Gr-ex} below).

The next result shows that the distribution of dimensions of $E_\infty^{pq}(X)$ for constant $p+q$, is skewed towards smaller $p$ compare to $E_\infty^{pq}(X_0)$.

\begin{prop} For every $m\ge 0$, $p\ge 0$,
$$\sum_{i\ge p} \dim E_{\infty}^{i,m-i}(X)\le \sum_{i\ge p} h^{i,m-i}(X_0)$$
(and this becomes an equality for $p=0$). Furthermore, for every $m$, the natural map
$$E_{\infty}^{0,m}(X)\to E_\infty^{0,m}(X_0)$$
is surjective.
\end{prop}

\begin{proof}
The first inequality follows immediately from the fact that the filtration $(F_p(X))$ on $H^n_{dR}(X)$ induced by the Hodge to de Rham spectral sequence satisfies
$$F_p(X)\sub F_p(X_0).$$
The second assertion also follows, as we can identify this map with
$$H^m_{dR}(X)/F_1H^m_{dR}(X)\to H^m_{dR}(X_0)/F_1H^m_{dR}(X_0).$$
\end{proof}

\begin{cor}\label{surj-Hi-cor} 
Assume the characteristic is zero. Then for every $i\ge 0$, the composition
$$E_\infty^{0,i}(X)\to H^i(X,\OO_X)\to H^i(X_0,\OO_{X_0})$$
is surjective.
\end{cor}

\begin{remark}
Similarly, in characteristic zero, for any $i\ge 0$ and any $m\ge 0$, the natural map of hypercohomology of truncated de Rham complexes,
$$H^i(X,[\OO_X\to \Om_X^1\to\ldots \to \Om_X^m])\to H^i(X_0,[\OO_{X_0}\to \Om_{X_0}^1\to\ldots \to \Om_{X_0}^m])$$
is surjective.
\end{remark}
%Hence, in characteristic zero we have inequalities

%0\to H^0_{dR}(X_0)\to H^0(X,\OO)\to H^0(X,\Om^1)\to H^1_{dR}(X_0)\to H^1(X,\OO)

\begin{prop}
Assume that $X$ is projected. Then for every $i\ge 1$, the natural maps $E_i^{pq}(X)\to H^q(\Om^p_{X_0})$ are surjective,
and they induce isomorphisms
$$E_\infty^{pq}(X)\rTo{\sim} H^q(\Om^p(X_0)).$$
In particular, $E_2^{10}(X)\simeq H^0(\Om^1_{X_0})$.
\end{prop}

\begin{proof} Let $i:X_0\to X$ denote the embedding and let $p:X\to X_0$ be a projection (so $p\circ i=\id_{X_0}$). Then we have 
morphisms of spectral sequences
$$E_n^{pq}(X_0)\rTo{p^*} E_n^{pq}(X)\rTo{i^*} E^n_{pq}(X_0),$$
with the composition equal to the identity.
The surjectivity immediately follows from this. The fact that we get an isomorphism on the infinite page
follows from the equality of dimensions $\dim E_\infty(X)=\sum_{pq} h^q(\Om^p_{X_0})$.
\end{proof}

%Something using class $c_1(L)$, where $L$ is ample???

\begin{prop}
(i) Let $L$ be an even line bundle on $X$. Then there exists a morphism of spectral sequences
$$c_1(L)\cup\cdot: E_n^{pq}(X)\to E_n^{p+1,q+1}(X),$$
such that the map on $E_1$ is given by the cup product with $c_1(L)\in H^1(\Om^1_X)$, compatible
with the cup product by $c_1(L)_{dR}\in H^2_{dR}(X)$ on the limit $H^\bullet_{dR}(X)$.
 
\noindent
(ii) Assume in addition that the characteristic is zero and $X$ admits an ample line bundle. 
Let $d=\dim X_0$. Then for any $i\ge j\ge 1$, and any $n\ge 1$, we have for $p+q=d-i$,
$$\rk(E_n^{pq}\to H^q(\Om^p_{X_0}))\le \rk(E_n^{p+j,q+j}\to H^{q+j}(\Om^{p+j}_{X_0})).$$
\end{prop}

\begin{proof}
(i) One can realize $c_1(L)$ by a Cech cocycle $d\log(f_{ij})$ with values in $\Om^{1,cl}$ (sheaf of closed $1$-forms), where $(f_{ij})$ are transition functions for $L$,
and use the multiplication by this cocycle on the Cech complex of $\Om^\bullet$.

\noindent
(ii) This follows from (i) together with the fact that the multiplication map
$$c_1(L)^j\cup\cdot: H^q(\Om^p_{X_0})\to H^{q+j}(\Om^{p+j}_{X_0})$$ 
is injective under our assumptions.
\end{proof}

%\begin{cor}
%Assume the characteristic is zero, $X$ admits an ample line bundle, and $ 
%Behavior under birational maps???

\begin{remark}
It is natural to ask whether the approach of Deligne-Illusie via the reduction to finite characteristic (see \cite{DI}) sheds any light on the Hodge to de Rham spectral
sequence in the super case. Namely, for $X/k$ in characteristic $p$,
the Frobenius kills odd functions, so it can be viewed as a morphism $F:X\to X_0$. Thus, $F_*\Om_X^{\bullet}$ becomes a complex of $\OO_{X_0}$-modules.
Now it is easy to see that the analog of the Cartier isomorphism (given by the same formulas) is
$$C^{-1}:\Om^i_{X_0}\rTo{\sim} \und{H}^i(F_*\Om_X^\bullet).$$
The same method as in \cite{DI} shows that if the characteristic is $>\dim X_0$ and 
$X$ lifts to characterstic zero, then the second spectral sequence for the hypercohomology of $F_*\Om_X^\bullet$ 
(starting from $E_2=\bigoplus H^q(\und{H}^i(F_*\Om_X^\bullet))$) degenerates at $E_2$. This agrees with Theorem \ref{deR-thm} but does not
give any additional information on the Hodge to de Rham spectral sequence for $X$ (because in the Cartier isomorphism $\Om^i_X$ gets replaced by $\Om^i_{X_0}$).
\end{remark}

\subsection{A result on even nilpotent extensions}\label{even-sec}

Corollary \ref{surj-Hi-cor} has the following classical counterpart (for usual schemes), which does not seem to be widely known. 

\begin{theorem}\label{nilp-sur-thm} 
Let $Y$ be a proper scheme over a field $k$ of characteristic zero, and let $Y_0$ be the corresponding reduced scheme.
Assume that $Y_0$ is smooth over $k$. Then for any $i$ the natural map $H^i(Y,\OO_Y)\to H^i(Y_0,\OO_{Y_0})$ is surjective.
\end{theorem}

\begin{proof}
The map in question fits into a commutative square of natural maps
\begin{diagram}
H^i(Y_{inf},\OO_Y)&\rTo{}&H^i((Y_0)_{inf},\OO_{Y_0})\\
\dTo{}&&\dTo{}\\
H^i(Y,\OO_Y)&\rTo{}& H^i(Y_0,\OO_{Y_0})
\end{diagram}
where in the first line we consider cohomology for the infinitesimal sites (see \cite{Gr}).
It is known that the top horizontal arrow is an isomorphism (see \cite[Sec.\ 5.3]{Gr}). Also,
it is well known that $H^i((Y_0)_{inf},\OO_{Y_0})$ is isomorphic to the de Rham cohomology of $Y_0$, so
%this map is an isomorphism. Also, it is known that H^i_{inf}(O_{X_0}) is isomorphic to the i-th de Rham cohomology,
by the Hodge-to-de Rham degeneration for $Y_0$, the right vertical arrow is surjective.
This implies that the composed map $H^i(Y_{inf},\OO_Y)\to H^i(Y_0,\OO_{Y_0})$ 
 is surjective. Since it factors through $H^i(Y,\OO_Y)$, the assertion follows.
\end{proof}

\section{Some computations}\label{comp-sec}

\subsection{Using a special derivation in the split case and in characterstic $2$}\label{split-sec}

\begin{theorem}\label{split-thm} 
Assume $X$ is split, and the ground field $k$ has characteristic zero. Then the Hodge to de Rham spectral sequence of $X$ degenerates at $E_2$ and
there are natural isomorphisms $E_2^{p,q}(X)\rTo{\sim} H^{p,q}(X_0)$.
\end{theorem}

\begin{proof}
Let $\OO_X=\bigwedge^\bullet(E)$, where $E$ is a bundle over $X_0$. Then we have an action of $\G_m$ on $X$, such that $\OO_{X_0}$ is $\G_m$-invariant,
and $E$ has weight $1$. Let $\xi$ be the corresponding global (even) vector field on $X$. In local coordinates $(x_\bullet,\th_\bullet)$ one has $\xi=\sum_i \th_i \partial_{\th_i}$.
Now we claim that for each $n\ge 0$, one has
$$\Om^n_X=\bigoplus_{i\ge 0} \Om^n_i,$$
where $\Om^n_i\sub \Om^n_X$ is the subsheaf of $\eta$ such that $L_\xi(\eta)=i \eta$, where $L_\xi$ is the Lie derivative with respect to $\xi$.
Indeed, this is a local statement, which is clear in local coordinates. Furthermore, for each $i\ge 0$, $\Om^\bullet_i$ is preserved by the de Rham differential (since $dL_\xi=L_\xi d$), and
$$\Om^\bullet_0\simeq \Om_{X_0}^\bullet.$$
Finally, the contraction operator $i_\xi$ gives a homotopy from $i\cdot \id$ to $0$ on $\Om^\bullet_i$. Hence, each complex $\Om^\bullet_i$ for $i>0$ is contractible.
This implies that the $E_1$ page of the spectral sequence associated with each $\Om^\bullet_i$ for $i>0$ is also contractible, so the $E_2$ page corresponding to each $\Om^\bullet_i$
with $i>0$ is zero. This implies our assertion.
\end{proof}

\begin{remark}
When $X$ varies in a flat family, there is no semicontinuity for the spaces $E_2^{p,q}(X)$ (since they are obtained by taking cohomology twice),
so the split case does not provide any nontrivial conclusions for the behavior of $E_2^{p,q}(X)$ for general $X$. For example, it is possible that for some supervariety $X$
one has $E_2^{p,q}(\gr X)=0$ for some $(p,q)$, while $E_2^{p,q}(X)\neq 0$ (where $\gr X$ is the split supervariety associated with $X$), see Example \ref{Gr-ex} below.
\end{remark}

When the characteristic of the ground field is $2$ then the vector field $\xi=\sum_i \th_i \partial_{\th_i}$ used in Theorem \ref{split-thm} does not depend on coordinates.
Namely, it corresponds to the derivation $\xi$ such that $\xi(f)=0$ for even $f$ and $\xi(\eta)=\eta$ for odd $\eta$.
This immediately leads to the following result.

\begin{prop}\label{char-2-prop}
Assume that the ground field $k$ has characterstic $2$. Then the odd part of the de Rham complex $(\Om^\bullet_X)_-$ is homotopic to zero. Hence, $E_2(X)_-=0$.
\end{prop}

\begin{proof} The homotopy from $\id$ to $0$ on $(\Om^\bullet_X)_-$ is given by the contraction operator $i_\xi$.
\end{proof}

\subsection{Odd dimension $2$}\label{n2-sec}

In this section we consider the case when $X$ is of dimension $n|2$. We assume that the characteristic of $k$ is zero.

%\begin{lemma} Assume that for every $i\ge 1$, the map $H^0(X,\NN^i)\to H^0(X,\NN^i/\NN^{i+1})$ is surjective.
%Then the natural map $E_2^{10}(X)\to H^{10}(X_0)=H^0(X_0,\Om^1)$ is injective.
%\end{lemma}

For each $p,q$, let us set
$$\de_{pq}(X):=\dim \coker(H^q(\Om^p_X)\to H^q(\Om^p_{X_0})).$$ 
Note that by Corollary \ref{surj-Hi-cor}, we have $\de_{0q}=0$.

\begin{theorem}\label{n2-dim-thm}
Assume that $X$ has dimension $n|2$, and the characteristic of $k$ is zero. 
Then 

\noindent
(i) one has $(E_2^{pq})_-=0$;

\noindent
(ii) the Hodge to de Rham spectral sequence of $X$ degenerates at $E_2$;

\noindent
(iii) one has
$$\dim E_2^{pq}=h^{pq}(X_0)+\de_{p+1,q-1}(X)-\de_{p,q}(X).$$
%$$\dim E_2^{10}(X)=h^{10}(X_0)-\de_{10}(X), \ \ \dim E_2^{01}(X)=h^{01}(X_0)+\de_{10}(X).$$
%$$\dim E_2^{20}(X)=h^{20}(X_0)-\de_{20}(X), \ \ 
%\dim E_2^{11}(X)=h^{11}(X_0)+\de_{20}(X)-\de_{11}(X), \ \ \dim E_2^{02}(X)=h^{02}(X_0)+\de_{11}(X).$$
%\noindent
%(iv) the natural maps
%$$E_2^{p0}(X)\to H^0(\Om^p_{X_0})$$
%are injective.
\end{theorem}

The proof will be based on the following result, where we use the filtration $F^\bullet\Om^\bullet_X$ introduced in the proof of Theorem \ref{deR-thm}.

\begin{prop}\label{n2-odd-thm}
Assume that $X$ has dimension $n|2$, and the characteristic of $k$ is zero.
Then the complexes $(\Om^\bullet_X)_-$ and $(F^2\Om^\bullet_X)_+$ are homotopic to zero.
\end{prop}

\begin{lemma}\label{split-lem} 
Assume we have a commutative diagram with exact rows in some abelian category,
\begin{diagram}
0&\rTo{}& A_1&\rTo{}& A_2&\rTo{}& A_3&\rTo{}& 0\\
&&\dTo{f_1}&&\dTo{f_2}&&\dTo{f_3}\\
0&\rTo{}& B_1&\rTo{}& B_2&\rTo{}& B_3&\rTo{}& 0
\end{diagram}
Assume that $f_1$ and $f_3$ are split injections and the second row is split exact, then $f_2$ is also a split injection.
\end{lemma}

The proof is straightforward and is left to the reader.

Below we will write $(\Om^\bullet_X)_\pm=\Om^\bullet_\pm$ for brevity. For $k\ge 2$, let us set 
$$G^k\Om^n_\pm:=(\NN^2 F^{k-2}\Om^n)_\pm + F^{k+2}\Om^n_\pm,$$
where the parity of $k$ is the same as the parity of the space. For $k=1$, we set $G^1\Om^n_-=F^3\Om^n_-$. 

\begin{lemma}\label{n2-odd-lem}
(i) For every $k\ge 1$
%odd $k\ge 1$ (resp., even $k\ge 2$) 
and every $n$, the following sequence is exact:
\begin{equation}\label{N2-Fk-main-seq}
0\to G^k\Om^n_\pm/d(F^{k+2}\Om^{n-1}_\pm)\to F^k\Om^n_\pm/d(F^k\Om^{n-1}_\pm)\rTo{\de}\Om^{n-k+1}_{X_0}\ot S^k(\NN/\NN^2)\to 0,
\end{equation}
where we use $+$ for even $k$ (resp., $-$ for odd $k$), and $\de$ is given by 
$$\de(\a_{n-k+1}\cdot n_1dn_2\cdot\ldots\cdot dn_k)=\a_{n-k+1}|_{X_0}\ot (\ov{n}_1\cdot\ov{n}_2\cdot\ldots\cdot\ov{n}_k),$$ 
with $\a_{n-k+1}\in \Om^{n-k+1}$.

\noindent
(ii) For every $k\ge 1$, the following map induced by $d$ is a split injection:
\begin{equation}\label{split-inj-main-square}
F^k\Om^n_\pm/d(F^k\Om^{n-1}_\pm)\to F^k\Om^{n+1}_\pm/\NN^2\cdot F^{k-2}\Om^{n+1}_\pm,
\end{equation}
where for $k=1$ we replace $\NN^2 F^{-1}\Om^{n+1}_-$ by zero, and we use $+$ for even $k$ (resp., $-$ for odd $k$),
% in the commutative square
%\begin{diagram}
%&\rTo{}& F^k\Om^n_-/(d(F^k\Om^{n-1}_-)+F^{k+2}\Om^n_-)\\
%\dTo{d}&&\dTo{d}\\
%&\rTo{}& F^k\Om^{n+1}_-/G^k\Om^{n+1}_-
%\end{diagram}
\end{lemma}

\begin{proof} We assume that $k\ge 1$ is odd. The case of even $k\ge 2$ is completely analogous.

\noindent
(i) We use the induction on $n\ge -1$ (for $n=-1$ the assertion is clear).
Let us write $S^i=S^i(\NN/\NN^2)$ for brevity.

Consider the commutative diagram with exact rows
%\begin{diagram}
%&&F^{k+2}\Om^{n-1}_-/d(F^{k+2}\Om^{n-2}_-)&\rTo{}&F^{k}\Om^{n-1}_-/d(F^k\Om^{n-2}_-)&\rTo{}& F^k\Om^{n-1}_-/(d(F^k\Om^{n-2}_-)+F^{k+2}\Om^{n-1}_-)&\rTo{}&0\\
%&&\dTo{d}&&\dTo{d}&&\dTo{d}\\
%0&\rTo{}& \NN^2F^{k-2}\Om^n_-+F^{k+2}\Om^n_-&\rTo{}&F^k\Om^n_-&\rTo{}&F^k\Om^{n}_-/(\NN^2\cdot F^{k-2}\Om^{n}_-+F^{k+2}\Om^{n}_-)&\rTo{}& 0
%\end{diagram}
%where the vertical arrows are induced by $d$. By the induction assumption, the rightmost vertical arrow is (split) injective. This implies that the map between the cokernels of
%the two left vertical arrows,
%$$(\NN^2F^{k-2}\Om^n_-+F^{k+2}\Om^n_-)/d(F^{k+2}\Om^{n-1}_-)\to F^k\Om^n_-/d(F^{k}\Om^{n-1}_-),$$
%is injective. Note that this is the left arrow in the sequence \eqref{}
%Next, let us consider the commutative diagram with exact rows 
(where the first row is exact by the induction assumption)
\begin{diagram}
%0&\rTo{}& (\NN^2F^{k-2}\Om^{n-1}_-+F^{k+2}\Om^{n-1}_-)/d(F^{k+2}\Om^{n-2}_-)&\rTo{}& F^k\Om^{n-1}_-/d(F^k\Om^{n-2}_-)&\rTo{\de}&\Om^{n-k}_{X_0}\ot S^k&\rTo{}& 0\\
0&\rTo{}& G^k\Om^{n-1}_-/d(F^{k+2}\Om^{n-2}_-)&\rTo{}& F^k\Om^{n-1}_-/d(F^k\Om^{n-2}_-)&\rTo{\de}&\Om^{n-k}_{X_0}\ot S^k&\rTo{}& 0\\
&&\dTo{d}&&\dTo{d}&&\dTo{\id}\\
%0&\rTo{}& \NN d\NN F^{k-2}\Om^{n-1}_-+\NN^2F^{k-2}\Om^n_-+F^{k+2}\Om^n_-&\rTo{}& F^k\Om^n_-&\rTo{}&\Om^{n-k}_{X_0}\ot S^k&\rTo{}& 0
0&\rTo{}& \NN\cdot F^{k-1}\Om^n_+ + G^k\Om^n_-&\rTo{}& F^k\Om^n_-&\rTo{}&\Om^{n-k}_{X_0}\ot S^k&\rTo{}& 0
\end{diagram} 
It shows that the natural embedding induces an isomorphism
\begin{equation}\label{NdN-Fk-isom}
%(\NN d\NN F^{k-2}\Om^{n-1}_-+\NN^2F^{k-2}\Om^n_-+F^{k+2}\Om^n_-)/d(\NN^2F^{k-2}\Om^{n-1}_-+F^{k+2}\Om^{n-1}_-)\rTo{\sim} F^k\Om^n_-/d(F^k\Om^{n-1}_-).
(\NN \cdot F^{k-1}\Om^{n}_+ +G^k\Om^n_-)/d(G^k\Om^{n-1}_-)\rTo{\sim} F^k\Om^n_-/d(F^k\Om^{n-1}_-).
\end{equation}
Next, we have a commutative diagram with exact rows,
\begin{diagram}
%0&\rTo{}& F^{k+2}\Om^{n-1}_-&\rTo{}& \NN^2F^{k-2}\Om^{n-1}_-+F^{k+2}\Om^{n-1}_-&\rTo{}& \Om_{X_0}^{n-k+1}\ot\NN^2\ot S^{k-2}&\rTo{}& 0\\
0&\rTo{}& F^{k+2}\Om^{n-1}_-&\rTo{}&  G^k\Om^{n-1}_-&\rTo{}& \Om_{X_0}^{n-k+1}\ot\NN^2\ot S^{k-2}&\rTo{}& 0\\
&&\dTo{d}&&\dTo{d}&&\dTo{\id\ot\kappa}\\
%0&\rTo{}& \NN^2F^{k-2}\Om^n_-+F^{k+2}\Om^n_-&\rTo{}& \NN d\NN F^{k-2}\Om^{n-1}_-+\NN^2F^{k-2}\Om^n_-+F^{k+2}\Om^n_-&\rTo{}& \Om_{X_0}^{n-k+1}\ot\NN/\NN^2\ot S^{k-1}&\rTo{}& 0
0&\rTo{}& G^k\Om^n_-&\rTo{}& \NN \cdot F^{k-1}\Om^{n}_+ +G^k\Om^n_-&\rTo{}& \Om_{X_0}^{n-k+1}\ot\NN/\NN^2\ot S^{k-1}&\rTo{}& 0
\end{diagram}
where $\kappa$ is the Koszul differential. Since $\kappa$ is injective, we deduce exactness of the sequence of the cokernels of the vertical maps. Taking into account the
isomorphism \eqref{NdN-Fk-isom}, we get exactness of
\eqref{N2-Fk-main-seq}.

\noindent
(ii) We use 
%the increasing induction on $n\ge -1$, and for each $n$, 
the decreasing induction on $k$. Note that $F^{n+3}\Om^n_-=0$, so the base of the induction is clear.
%For $n=-1$ the statement is clear.
%Now fix $(n,k)$ with $n\ge 0$, $k\ge 1$ odd, and assume the assertion holds for $(n',k')$ with $n'<n$ and for $(n,k')$ with $k'>k$.

Assume first that $k\ge 3$.
Consider the commutative diagram with exact rows,
\begin{equation}\label{F-G-modN2-eq1}
\begin{diagram}
%0&\rTo{}& (\NN^2F^{k-2}\Om^n_-+F^{k+2}\Om^n_-)/d(F^{k+2}\Om^{n-1}_-)&\rTo{}& F^k\Om^n_-/d(F^k\Om^{n-1}_-)&\rTo{}&\Om^{n-k+1}_{X_0}\ot S^k&\rTo{}& 0\\
0&\rTo{}& G^k\Om^n_-/d(F^{k+2}\Om^{n-1}_-)&\rTo{}& F^k\Om^n_-/d(F^k\Om^{n-1}_-)&\rTo{\de}&\Om^{n-k+1}_{X_0}\ot S^k&\rTo{}& 0\\
&&\dTo{d}&&\dTo{d}&&\dTo{\id}\\
%0&\rTo{}& (\NN d\NN F^{k-2}\Om^n_-+\NN^2F^{k-2}\Om^{n+1}_-+F^{k+2}\Om^{n+1}_-)/\NN^2 F^{k-2}\Om^{n+1}_-&\rTo{}& 
%F^k\Om^{n+1}_-/\NN^2 F^{k-2}\Om^{n+1}_-&\rTo{}&\Om^{n-k+1}_{X_0}\ot S^k&\rTo{}& 0
0&\rTo{}& (\NN \cdot F^{k-1}\Om^{n+1}_+ +G^k\Om^{n+1}_-)/\NN^2 F^{k-2}\Om^{n+1}_-&\rTo{}& 
F^k\Om^{n+1}_-/\NN^2 F^{k-2}\Om^{n+1}_-&\rTo{}&\Om^{n-k+1}_{X_0}\ot S^k&\rTo{}& 0
\end{diagram}
\end{equation}
%We also have a similar diagram in which we take quotients of the first two terms in the first row (resp., second row) by the image of $F^{k+2}\Om^n_-$
%(resp., $F^{k+2}\Om^{n+1}_-$). 
Using Lemma \ref{split-lem} (with $f_3=\id$),
we see that it is enough to prove that the left vertical arrow
%we can replace the map \eqref{split-inj-main-square} by 
%\begin{equation}\label{F-G-modN2-eq2}
% G^k\Om^n_-/d(F^{k+2}\Om^{n-1}_-) \to (\NN \cdot F^{k-1}\Om^{n+1}_+ +G^k\Om^{n+1}_-)/\NN^2 F^{k-2}\Om^{n+1}_-
%\begin{diagram}
% &\rTo{}& G^k\Om^n_-/(d(F^{k+2}\Om^{n-1}_-)+F^{k+2}\Om^n_-)\\
%\dTo{d}&&\dTo{d}\\
%&\rTo{}& 
%(\NN \cdot F^{k-1}\Om^{n+1}_+ +G^k\Om^{n+1}_-)/(\NN^2 F^{k-2}\Om^{n+1}_-+F^{k+2}\Om^{n+1}_-)
%\end{diagram}
%\end{equation}
is a split injection. But the same map appears as the middle vertical arrow in the commutative diagram with exact rows,
\begin{equation}\label{last-F-G-diagram}
\begin{diagram}
%0&\rTo{}& F^{k+2}\Om^n_-/d(F^{k+2}\Om^{n-1}_-)&\rTo{}& (\NN^2F^{k-2}\Om^n_-+F^{k+2}\Om^n_-)/d(F^{k+2}\Om^{n-1}_-)&\rTo{}&\Om^{n-k+2}_{X_0}\ot \NN^2\ot S^{k-2}&\rTo{}&0\\
0&\rTo{}& F^{k+2}\Om^n_-/d(F^{k+2}\Om^{n-1}_-)&\rTo{}& G^k\Om^n_-/d(F^{k+2}\Om^{n-1}_-)&\rTo{}&\Om^{n-k+2}_{X_0}\ot \NN^2\ot S^{k-2}&\rTo{}&0\\
&&\dTo{d}&&\dTo{d}&&\dTo{\id\ot\kappa}\\
%0&\rTo{}& F^{k+2}\Om^{n+1}_-/\NN^2F^k\Om^{n+1}_-&\rTo{}&(\NN d\NN F^{k-2}\Om^n_-+\NN^2F^{k-2}\Om^{n+1}_-+F^{k+2}\Om^{n+1}_-)/\NN^2 F^{k-2}\Om^{n+1}_-
%&\rTo{}&\Om^{n-k+2}_{X_0}\ot \NN/\NN^2\ot S^{k-1}&\rTo{}&0
0&\rTo{}& F^{k+2}\Om^{n+1}_-/\NN^2F^k\Om^{n+1}_-&\rTo{}&(\NN \cdot F^{k-1}\Om^{n+1}_+ +G^k\Om^{n+1}_-)/\NN^2 F^{k-2}\Om^{n+1}_-
&\rTo{}&\Om^{n-k+2}_{X_0}\ot \NN/\NN^2\ot S^{k-1}&\rTo{}&0
\end{diagram}
\end{equation}
The leftmost vertical arrow is split injective by the induction assumption, and the rightmost vertical arrow is split injective since the Koszul differential $\kappa$ is so.
By Lemma \ref{split-lem}, it remains to prove that the second row is split exact. But the splitting is induced by the well defined map
\begin{align*}
&\Om^{n-k+2}_{X_0}\ot \NN/\NN^2\ot S^{k-1}\to \NN \cdot F^{k-1}\Om^n_+/\NN^2 F^k\Om^n_-: \\
&df_1\ldots df_{n-k+2}\ot n_1\ot (n_2\ldots n_k)\mapsto d\wt{f}_1\ldots d\wt{f}_{n-k+2} n_1 dn_2 \ldots dn_k,
\end{align*}
where $\wt{f}_i\in (\OO_X)_+$ are liftings of $f_i\in \OO_{X_0}$.

In the case $k=1$, we modify the above argument as follows: we replace $\NN^2F^{-1}\Om^{n+1}_-$ by zero
in the diagram \eqref{F-G-modN2-eq1} and replace the diagram \eqref{last-F-G-diagram} by
\begin{diagram}
%0&\rTo{}& F^{k+2}\Om^n_-/d(F^{k+2}\Om^{n-1}_-)&\rTo{}& (\NN^2F^{k-2}\Om^n_-+F^{k+2}\Om^n_-)/d(F^{k+2}\Om^{n-1}_-)&\rTo{}&\Om^{n-k+2}_{X_0}\ot \NN^2\ot S^{k-2}&\rTo{}&0\\
0&\rTo{}& F^3\Om^n_-/d(F^3\Om^{n-1}_-)&\rTo{\id}& F^3\Om^n_-/d(F^3\Om^{n-1}_-)&\rTo{}& 0\\
&&\dTo{d}&&\dTo{d}\\
%0&\rTo{}& F^{k+2}\Om^{n+1}_-/\NN^2F^k\Om^{n+1}_-&\rTo{}&(\NN d\NN F^{k-2}\Om^n_-+\NN^2F^{k-2}\Om^{n+1}_-+F^{k+2}\Om^{n+1}_-)/\NN^2 F^{k-2}\Om^{n+1}_-
%&\rTo{}&\Om^{n-k+2}_{X_0}\ot \NN/\NN^2\ot S^{k-1}&\rTo{}&0
0&\rTo{}& F^3\Om^{n+1}_-/\NN^2 \Om^{n+1}_-&\rTo{}&(\NN \cdot \Om^{n+1}_+ +F^3\Om^{n+1}_-)/\NN^2 \Om^{n+1}_-
&\rTo{}&\Om^{n+1}_{X_0}\ot \NN/\NN^2&\rTo{}&0
\end{diagram}
As before, we see that the bottom row is split exact, and the leftmost vertical arrow is split injective by the induction assumption. Hence, the middle vertical arrow is split exact,
which implies the assertion.
\end{proof}

\begin{proof}[Proof of Proposition \ref{n2-odd-thm}]
For $k=1$ (resp., $k=2$), Lemma \ref{n2-odd-lem}(ii) implies that the map $\Om^n_-/d\Om^{n-1}_-\to \Om^{n+1}_-$
(resp., $F^2\Om^n_+/dF^2\Om^{n-1}_+\to F^2\Om^{n+1}_+$) is a split injection,
which immediately gives the result.
\end{proof}

\begin{proof}[Proof of Theorem \ref{n2-dim-thm}]
Part (i) immediately follows from Proposition \ref{n2-odd-thm}. 
To prove parts (ii) and (iii), it is enough to prove the inequalities 
$$\dim E_2^{pq}=\dim (E_2^{pq})_+\le h^{pq}(X_0)+\de_{p+1,q-1}(X)-\de_{p,q}(X).$$
Indeed, summing up these inequalities, we would get
$$\dim E_2(X)\le \dim H_{dR}(X_0)=\dim E_\infty(X),$$
hence the spectral sequence degenerates at $E_2$ and the above inequalites are in fact equalities.

Since the natural map $(E_2^{pq})_+\to H^q(\Om^p_{X_0})$ factors through a quotient of $H^q(\Om^p_X)$, it has corank $\ge \de_{p,q}(X)$. Thus, it
suffices to show that
$$\dim \ker((E_2^{pq})_+\to H^q(\Om^p_{X_0}))\le \de_{p+1,q-1}(X).$$
The exact sequence
$$0\to F^2(\Om^p_X)_+\to (\Om^p_X)_+\to \Om^p_{X_0}\to 0$$
gives rise to the surjection
$$K:=\ker(H^q(F^2(\Om^p_X)_+)\to H^q((\Om^{p+1}_X)_+))/\im H^q((F^2\Om^{p-1}_X)_+)\to \ker((E_2^{pq})_+\to H^q(\Om^p_{X_0})).$$
%\im H^{q-1}(\Om^1_{X_0})
Let us consider the morphism of exact sequences
\begin{diagram}
0&\rTo{}& F^2(\Om^p_X)_+&\rTo{}& (\Om^p_X)_+&\rTo{}& \Om^p_{X_0}&\rTo{}& 0\\
&&\dTo{d}&&\dTo{d}&&\dTo{d}\\
0&\rTo{}& F^2(\Om^{p+1}_X)_+ &\rTo{}& (\Om^{p+1}_X)_+ &\rTo{}& \Om^{p+1}_{X_0} &\rTo{}& 0
\end{diagram}
Using exactness of the sequence
$$H^q(F^2(\Om^{p-1}_X)_+)\to H^q(F^2(\Om^p_X)_+)\to H^q(F^2(\Om^{p+1}_X)_+),$$
which follows from Proposition \ref{n2-odd-thm}, 
we get an embedding
$$K\hra \ker(H^q(F^2(\Om^{p+1}_X)_+)\to H^q((\Om^{p+1}_X)_+))\simeq \coker(H^{q-1}((\Om^{p+1}_X)_+)\to H^{q-1}(\Om^{p+1}_{X_0}))$$
so $\dim K\le \de_{2,q-1}(X)$.
%It is enough to show exactness of the sequence
%$$H^q(\NN^2)\to H^q(F^2(\Om^1_X)_+)\to H^q((\Om^2_X)_+).$$
%together with the vanishing of $H^{q-1}(\Om^2_{X_0})$ shows that the map $H^q(F^2(\Om^2_X)_+)\to H^q((\Om^2_X)_+)$ is injective.
\end{proof}

\begin{cor}
If $X$ has dimension $n|2$ then $E_2^{pq}=0$ for $p>n$ (and clearly for $q>n$).
\end{cor}

Our next goal is to get a lower bound on $\de_{1q}(X)$ (see Proposition \ref{de-mu-prop} below).
Set $L:=\NN^2$ viewed as a line bundle on $X_0$. 
Let $Y$ denote the even scheme with the same underlying space as $X_0$ and with the structure sheaf $\OO_Y=(\OO_X)_+$. Then $\OO_Y$ is a square zero extension of $\OO_{X_0}$ by $L$, so it corresponds to a class 
$e\in H^1(T_{X_0}\ot L)$.

\begin{lemma}\label{Om-Y-X0-lem}
One has a natural exact sequence of coherent sheaves on $X_0$,
\begin{equation}\label{Om1-extension}
0\to L\rTo{\iota} \Om^1_Y|_{X_0}\to \Om^1_{X_0}\to 0
\end{equation}
and the corresponding extension class corresponds to $e$ under the natural isomorphism $\Ext^1(\Om^1_{X_0},L)\simeq H^1(T_{X_0}\ot L)$.
\end{lemma}

\begin{proof} This exact sequence is standard: its exactness on the left follows from the fact that locally $Y$ is a Cartier divisor in $X_0\times \A^1$.
Let $(U_i)$ be an open affine covering of $X_0$ and let $\si_i:\OO_{X_0}\to \OO_Y$ be splittings defined over $U_i$, so that
the class of $e$ is represented by the cocylce $v_{ij}\in T_{X_0}\ot L(U_{ij})$, such that
$$(\si_j-\si_i)(f)=\iota\lan v_{ij},df\ran.$$
We have the induced 
splittings over $U_i$,
$$d\si_i:\Om^1_{X_0}\to \Om^1_Y|_{X_0}: fdg\mapsto fd\si_i(g)|_{X_0}.$$
Now the cocycle $d\si_j-d\si_i\in (\Om^1)^\vee\ot L(U_{ij})$ represents the class of the extension \eqref{Om1-extension}.
It remains to observe that
$$\lan d\si_j-d\si_i,dg\ran=d(\si_j(g)-\si_i(g))=\iota\lan v_{ij},df\ran,$$
so this cocycle coincides with $v_{ij}$.
\end{proof}

\begin{lemma}\label{Om-X-Y-X0-lem}
(i) One has a canonical decomposition
$$(\Om^1_X)_+/(\NN^2\ot\Om^1_{X_0})\simeq \Om^1_Y|_{X_0}\oplus S^2(\NN/\NN^2).$$

\noindent
(ii) For any $q\ge 0$, one has an embedding
$$\im(H^q(\Om^1_X)\to H^q(\Om^1_{X_0}))\sub \im(H^q(\Om^1_Y|_{X_0})\to H^q(\Om^1_{X_0})).$$
\end{lemma}

\begin{proof}
(i) The pull-back map $\Om^1_Y\to (\Om^1_X)_+$ induces an $\OO_{X_0}$-linear map $\Om^1_Y|_{X_0}\to (\Om^1_X)_+/(\NN^2\ot\Om^1_{X_0})$.
On the other hand, the map 
$$\NN_-\ot \NN_-\to (\Om^1_X)_+: a\ot b\mapsto adb+bda$$
induces an $\OO_{X_0}$-linear map $S^2(\NN/\NN^2)\to (\Om^1_X)_+/(\NN^2\ot\Om^1_{X_0})$.
One can check using local coordinates that these maps are embeddings and give the claimed decomposition.

\noindent
(ii) Since the map $(\Om^1_X)_+\to \Om^1_{X_0}$ factors through the quotient $(\Om^1_X)_+/(\NN^2\ot\Om^1_{X_0})$, this follows from
the decomposition proved in part (i).
\end{proof}

\begin{prop}\label{de-mu-prop} 
Consider the cup-product map 
$$\mu^q_e:H^q(\Om^1_{X_0})\rTo{\cup e} H^{q+1}(L).$$
Then one has
$$\de_{1q}(X)\ge \rk(\mu^q_e).$$
\end{prop}

\begin{proof}
By Lemma \ref{Om-X-Y-X0-lem}(ii), 
$$\de_{1q}(X)\ge \dim \coker(H^q(\Om^1_Y|_{X_0})\to H^q(\Om^1_{X_0})).$$
But by Lemma \ref{Om-Y-X0-lem}, we have an identification
$$\coker(H^q(\Om^1_Y|_{X_0})\to H^q(\Om^1_{X_0}))\simeq \im(\mu^q_e).$$
\end{proof}

\begin{cor} Assume that $X_0$ is a smooth projective curve of genus $g\ge 1$, and $L\simeq \om_{X_0}$. Then $\de_{10}(X)\ge 1$.
\end{cor}

\begin{example}\label{Gr-ex}
Consider $X=G(1|1,2|2)$. Then $X_0=\P^1\times \P^1$, and $\NN/\NN^2\simeq \OO(-1,-1)^{\oplus 2}$, $\NN^2\simeq \om_{X_0}=\OO(-2,-2)$.
%so the conditions of Proposition \ref{n2-deg-prop-2} are satisfied. 
%Hence, $\dim E_2^{11}(X)=2-\de_{11}(X)$.
By Proposition \ref{de-mu-prop}, we have 
$$\de_{11}(X)\ge \rk(\mu^1_e:H^1(\Om^1_{X_0})\to H^2(\Om^2_{X_0})).$$
It is well known that $e\in H^1(T_{X_0}\ot \om_{X_0})\simeq H^1(\Om^1_{X_0})$ is nonzero, hence $\mu^1_e\neq 0$, and we get $\de_{11}(X)\ge 1$.
On the other hand, let $\BB$ denote the Berezinian of the tautological bundle on $X=G(1|1,2|2)$. Then $c_1(\BB)$ gives an element of $H^1(\Om^1_X)$
projecting to a nonzero class in $H^1(\Om^1_{X_0})$ (see e.g., \cite[Ch.\ 4.3]{Manin}). Hence, we have $\de_{11}(X)=1$, and so
$$\dim E_2^{11}(X)=\dim E_2^{02}(X)=1.$$

In fact, it is easy to see that $H^i(\Om_X^p)=0$ for $p\ge 2$ and $i=0,1$, which implies that
%in this example the spectral sequence degenerates at $E_2$. Furthermore, by Lemma \ref{E2-1q-lem} with $q=2$, we get $E_2^{12}=0$.
%Hence,
the only nonzero $E_2$ terms are
$$E_2^{00}=E_2^{02}=E_2^{11}=E_2^{22}=k.$$
\end{example}

\subsection{Example in dimension $1|3$}\label{13-dim-sec}

We will give an example showing that in dimension $1|3$ the map on $H^1$ induced by $d:(\OO_X)_-\to (\Om^1_X)_-$ may have a nontrivial kernel
(see Theorem \ref{nondeg-E2-thm} below). 

We start with a general construction of smooth supervarieties of dimension $1|3$.

\begin{lemma}\label{13-constr-lem}
Let $C$ be a smooth curve, $V$ a vector bundle of rank $3$ over $C$, 
\begin{equation}\label{sq-zero-ext-Y}
0\to {\bigwedge}^2 V\to \OO_Y\to \OO_C\to 0
\end{equation}
a square zero extension. Then there exists a smooth supervariety $X$ of dimension $1|3$ with $X_0=C$, $\NN/\NN^2=V$, and $(\OO_X)_+\simeq \OO_Y$
as a square zero extension of $\OO_C$.
\end{lemma}

\begin{proof}
Let $e\in H^1(T_C\ot{\bigwedge}^2 V)$ be the class of the extension \eqref{sq-zero-ext-Y}. Note that we have a natural isomorphism 
$${\bigwedge}^2 V\simeq V^\vee\ot \det(V)\simeq\und{\Hom}(V,\det(V)).$$
Let $D_{\le 1}(V,\det(V))$ denote the sheaf of differential operators of order $\le 1$ from $V$ to $\det(V)$. Then we have an exact sequence
$$0\to \und{\Hom}(V,\det(V))\to D_{\le 1}(V,\det(V))\rTo{\si} T_C\ot \und{\Hom}(V,\det(V))\to 0,$$
where $\si$ is the symbol map. Since $H^2(C,\und{\Hom}(V,\det(V)))=0$, we can lift $e$ to a class 
$\wt{e}\in H^1(C, D_{\le 1}(V,\det(V)))$. Let us represent $\wt{e}$ by a Cech cocycle with respect to an affine covering $C=(U_1,U_2)$,
$\phi_{12}:V|_{U_{12}}\to \det(V)|_{U_{12}}$. By assumption, it satisfies
$$\phi_{12}(f s)=f\phi_{12}(s)+v_{12}(f)\we s,$$
for $f\in \OO$, $s\in V$, where $v_{12}\in H^0(U_{12},T_C\ot {\bigwedge}^2 V)$ is the Cech cocycle representing $e$.

Now we are going to glue $X$ from the split supervarieties over $U_i$, $i=1,2$, corresponding to $\bigwedge^\bullet(V|_{U_i})$.
We just need to construct an automorphism $\a$ of the sheaf of $\Z/2$-graded rings $\bigwedge^\bullet(V|_{U_{12}})$ inducing the identity
modulo $\bigwedge^{\ge 1}$. It is enough to specify how $\a$ acts on $\OO$ and on $V$. We set for $f\in \OO$, $s\in V$,
$$\a(f)=f+v_{12}(f), \ \ \a(s)=s+\phi_{12}(s).$$
It is easy to check that this gives a well defined automorphism of $\bigwedge^\bullet(V|_{U_{12}})$,
and thus, defines the supervariety $X$ with the required properties.
\end{proof}

\begin{theorem}\label{nondeg-E2-thm}
Assume that the ground field $k$ is algebraically closed field of characteristic $\neq 2,3$.
For any smooth projective curve $C$ of genus $2$ over $k$, 
there exists a smooth supervariety $X$ of dimension $1|3$ with $X_0=C$, and such that 
$$E_2^{01}(X)_-=\ker(H^1(X,(\OO_X)_-)\to H^1(X,(\Om^1_X)_-))\neq 0.$$
In particular, the Hodge to de Rham spectral sequence of $X$ does not degenerate at $E_2$.
\end{theorem}

We start with some auxiliary statements.

\begin{lemma}\label{exterior-power-lem}
Assume the characteristic of $k$ is $\neq 2,3$.
For a $k$-vector space $V$, the linear map
$$\tau_V: V\ot {\bigwedge}^2 V\to V\ot {\bigwedge}^2 V: v_1\ot (v_2\we v_3)\mapsto v_3\ot (v_1\we v_2)- v_2\ot (v_1\we v_3)$$
is invertible.
\end{lemma}

\begin{proof}
Let us consider the projector $\pi$ on $V\ot {\bigwedge}^2 V$ (onto a subspace isomorphic to ${\bigwedge}^3 V$) given by
$$v_1\ot (v_2\we v_3)\mapsto \frac{1}{3}\bigl(v_1\ot (v_2\we v_3)+v_2\ot (v_3\we v_1)+v_3\ot (v_1\we v_2)\bigr).$$
Then we have $\tau_V=3\pi-\id$, so it has eigenvalues $2$ and $-1$.
\end{proof}

\begin{lemma}\label{rk3-bundle-lem}
Assume the characteristic of $k$ is $\neq 2,3$.
Let $C$ be a smooth projective curve over $k$, and let $V$ be a rank $3$ vector bundle on $C$ such that
$$H^1(\und{\End}_0(V)\ot \det(V))=H^0(V)=H^0(\det(V)^{-1})=0,$$ 
while $H^0(\om_C\ot V)\neq 0$ and $H^1(\det(V))\neq 0$.
Then there exists a smooth supervariety $X$ of dimension $1|3$ with $X_0=C$, and $\NN/\NN^2=V$, such that $E_2^{01}(X)_-\neq 0$.
\end{lemma}

\begin{proof}
By Lemma \ref{13-constr-lem}, for any class $e\in H^1(T_C\ot \bigwedge^2 V)$, there exists a supervariety $X$ 
of dimension $1|3$ with $X_0=C$, and $\NN/\NN^2=V$, such that the extension class of $(\OO_X)_+\to \OO_C$ is $e$.
We will show that there exists a choice of $e$ such that $E_2^{01}(X)_-\neq 0$.

The exact sequence
$$0\to \NN^3\to\OO_-\to V\to 0$$
shows that the map $H^1(\NN^3)\to H^1(\OO_-)$ is injective. Thus, it is enough to prove
that the map $H^1(\NN^3)\to H^1((\NN\Om^1)_-)$ has a nontrivial kernel.
Since $\NN^3\Om^1_{X_0}\simeq \det(V)\ot \om_C$ has trivial $H^1$ (which is dual to $H^0(\det(V)^{-1})=0$), it is enough to prove that the map
$$H^1(\NN^3)\to H^1((\NN\Om^1)_-/\NN^3\Om^1_{X_0})$$
has a nontrivial kernel.

We have an exact sequence of $\OO_{X_0}$-modules
\begin{equation}\label{N2dN-ex-seq}
0\to \NN^2d\NN\to (\NN\Om^1)_-/\NN^3\Om^1_{X_0}\to \NN/\NN^2\ot \Om^1_{X_0}\to 0.
\end{equation}
Note that if $\si_i:\OO_{X_0}\to (\OO_X)_+$ are local splittings of the extension $(\OO_X)_+\to \OO_C$ then we have induced splittings of 
\eqref{N2dN-ex-seq} given by $n\ot df\mapsto \wt{n}\cdot d\si_i(f)$ (where $\wt{n}\in \NN_-$ is a lifting of $n\in \NN/\NN^2$). The difference between two such splittings
is given by the $1$-cocycle $n\ot df\mapsto \wt{n}\cdot d\lan v_{ij},df\ran$, where $v_{ij}$ is the $1$-cocycle representing $e\in H^1(T_{X_0}\ot {\bigwedge}^2 V)$.
Using the Leibnitz rule we deduce that the extension class of \eqref{N2dN-ex-seq} in $\Ext^1(V\ot \Om^1_{X_0}, V\ot \bigwedge^2 V)$ is equal to the image of $e$
under the natural map
$$H^1(T_{X_0}\ot {\bigwedge}^2 V)\simeq \Ext^1(\Om^1_{X_0}, {\bigwedge}^2 V)\to \Ext^1(V\ot \Om^1_{X_0}, V\ot {\bigwedge}^2 V)\rTo{\tau_V\circ}
\Ext^1(V\ot \Om^1_{X_0}, V\ot {\bigwedge}^2 V),$$
%of the class $e\in H^1(T_{X_0}\ot {\bigwedge}^2 V)$.
where the last arrow is induced by the automorphism $\tau_V$ (see Lemma \ref{exterior-power-lem}).

We have an isomorphism $\bigwedge^2 V\simeq V^\vee\ot \det(V)$, so we have a direct sum decomposition
$$V\ot {\bigwedge}^2 V\simeq V\ot V^\vee \ot \det(V)\simeq \det(V)\oplus \und{\End}_0(V)\ot \det(V).$$
By the assumption, $H^1(\und{\End}_0(V)\ot \det(V))=0$, so the map
$$H^1(\NN^3)\simeq H^1(\det(V))\to H^1(V\ot {\bigwedge}^2 V)$$
is an isomorphism. Thus, it is enough to check that the connecting homomorphism 
$$H^0(\om_C\ot V)\to H^1(V\ot {\bigwedge}^2 V)$$
associated with \eqref{N2dN-ex-seq} is nonzero. Since the projection $V\ot \bigwedge^2 V\to \det(V)$ induces an isomorphism on $H^1$,
this is equivalent to showing that the homomorphism
$$H^0(\om_C\ot V)\rTo{\cup e} H^1(\det(V))$$
is nonzero, where we view $e$ as an element of $H^1(\om_C^{-1}\ot V^\vee\ot \det(V))$.
Thus, we need to choose $e$ such that this cup product map is nonzero.

This is possible as long as the map
$$H^0(\om_C\ot V)\ot H^1(\om_C^{-1}\ot V^{\vee}\ot \det(V))\to H^1(\det(V))$$
is nonzero. By Serre duality, we need to check that the map
$$H^0(\om_C\ot V)\ot H^0(\om_C\ot \det(V)^{-1})\to H^0(\om_C^2\ot V\ot \det(V)^{-1})$$
is nonzero. By assumption, there exists a nonzero element $s\in H^0(\om_C\ot \det(V)^{-1})$ (this space is dual to $H^1(\det(V))\neq 0$).
The multiplication by $s$ is an injective map
$$H^0(\om_C\ot V)\rTo{s} H^0(\om_C^2\ot V\ot \det(V)^{-1}).$$
Hence, our assertion follows from non-vanishing of $H^0(\om_C\ot V)$.
%\begin{diagram}
%0&\rTo{}&\NN^3&\rTo{}&\OO_-&\rTo{}&V&\rTo{}&0\\
%&&\dTo{d}&&\dTo{d}&&\dTo{\id}\\
%0&\rTo{}&(\NN\Om^1)_-&\rTo{}&\Om^1_-&\rTo{}&\NN/\NN^2&\rTo{}&0
%\end{diagram}
% injects into $E_2^{01}(X)_-=\ker(H^1(\OO_-)\to H^1(\Om^1_-))$.
\end{proof}

We will need some facts about analogs of Prym varieties associated with cyclic coverings (where we assume that the characteristic is $\neq 3$).

\begin{lemma}\label{prym-lem}
Let $\pi:\wt{C}\to C$ be a cyclic covering associated with a line bundle $\xi$ of order $3$ on a smooth projective curve $C$,
and let $\tau:\wt{C}\to \wt{C}$ be an action of a nontrivial element of the Galois group $\Z/3\Z$. 

\noindent
(i) The morphisms $\pi^*:J_C\to J_{\wt{C}}$ and $\Nm_{\wt{C}/C}:J_{\wt{C}}\to J_C$ are dual to each other with respect to
the canonical principal polarizations on the Jacobians.
One has 
$$\Nm_{\wt{C}/C}\circ \pi^*=[3]_{J_C}, \ \ \pi^*\circ\Nm_{\wt{C}/C}=\tau_*+\tau^{-1}_*+\id_{J_{\wt{C}}}.$$

\noindent
(ii) One has $\ker(\pi^*)=\lan \xi\ran\sub J_C$, the subgroup of order $3$ generated by $\xi$.

\noindent
(iii) Set 
$$\wt{P}=\ker(\Nm_{\wt{C}/C}:J_{\wt{C}}\to C),$$ 
and let $P\sub \wt{P}$ denote the connected component of $0$ in $\wt{P}$ ($P$ is an analog of the Prym variety).
Then $P=(\tau-\id)(J_{\wt{C}})$, $\tau-\id:P\to P$ is an isogeny,
%$P=\im(\tau_*-\id:J_{\wt{C}}\to J_{\wt{C}})$ , 
and $|\wt{P}/P|=3$.
\end{lemma}

\begin{proof}
(i) This is well known (see \cite[Sec.\ 11.4]{BL}).

\noindent 
(ii) Recall that $\pi_*\OO_{\wt{C}}\simeq \OO_C\oplus \xi\oplus \xi^{-1}$. 
Hence, 
$$h^0(\pi^*\xi)=h^0(\xi\ot \pi_*\OO)=h^0(\pi_*\OO)=1,$$
so $\pi^*\xi$ is a degree $0$ line bundle on $\wt{C}$ with a nonzero section, hence $\pi^*\xi\simeq \OO_{\wt{C}}$.

Conversely, assume $\pi^*M\simeq \OO$. Then 
$$M\ot\pi_*\OO\simeq \pi_*\pi^*M\simeq \pi_*\OO,$$ 
so $M$ is isomorphic to one of the summands in $\pi_*\OO$.

\noindent
(iii) Set $A:=J_{\wt{C}}/P$. We have an exact sequence of abelian varieties
\begin{equation}\label{ab-var-seq}
0\to P\to J_{\wt{C}}\rTo{p} A\to 0,
\end{equation}
where $p$ is induced by $\Nm_{\wt{C}/C}$,
and an isogeny $f:A\to J_C$, such that $\ker(f)\simeq \wt{P}/P$ and $\Nm_{\wt{C}/C}=f\circ p$.
Consider the dual isogeny $\hat{f}:J_C\to \hat{A}$, and the dual sequence to \eqref{ab-var-seq},
$$0\to \hat{A}\to J_{\wt{C}}\to \hat{P}\to 0.$$
Then the morphism $\pi^*:J_C\to J_{\wt{C}}$, being dual to  $\Nm_{\wt{C}/C}=f\circ p$, factors as the composition
of $\hat{f}$ followed by the embedding $\hat{A}\to J_{\wt{C}}$. Hence, we have 
$$\ker(\hat{f})=\ker(\pi^*)=\lan \xi\ran.$$
But $\ker(f)$ is Cartier dual to $\ker(\hat{f})$, so we deduce
$$|\wt{P}/P|=|\ker(f)|=|\ker(\hat{f})|=3.$$

Note that $\Nm_{\wt{C}/C}\circ\tau_*=\Nm_{\wt{C}/C}$, so $(\tau_*-\id)(J_{\wt{C}})$ is contained in $\wt{P}$, and hence, in $P$.
It remains to show that the morphism $(\tau_*-\id):P\to P$ is an isogeny. It is enough to show that the induced map on the tangent space $T_0P$ is nondegenerate.
We have 
$$T_0P\simeq \ker(H^1(\wt{C},\OO_{\wt{C}})\rTo{\tr} H^1(C,\OO_C)).$$
The decomposition $\pi_*\OO_{\wt{C}}=\OO_C\oplus \xi\oplus \xi^{-1}$ is compatible with the action of $\tau$: it corresponds to the eigenvalues $1$, $\zeta$ and $\zeta^{-1}$, where
$\zeta$ is a primitive $3$rd root of unity. Hence $\tau$ acts on $T_0P$ with eigenvalues $\zeta$ and $\zeta^{-1}$, which implies that $\tau-\id$ is nondegenerate.
\end{proof}

\begin{proof}[Proof of Theorem \ref{nondeg-E2-thm}]
It is enough to construct a rank $3$ vector bundle $V$ on $C$ satisfying the assumptions of Lemma \ref{rk3-bundle-lem}.

\noindent
{\bf Step 1. A triple covering and a special line bundle of order $3$}.
We start by choosing a line bundle $\xi$ of order $3$ on $C$. Let $\pi:\wt{C}\to C$ be the corresponding cyclic covering (where $\wt{C}$ has genus $4$). 
We will use the notations and the results of Lemma \ref{prym-lem}.
We claim that there exists a line bundle $\eta$ of order $3$, such that $\eta\not\in\lan \xi\ran$, and $\pi^*\eta\in P$.

Indeed, since $\Nm_{\wt{C}/C}\circ \pi^*=[3]_{J_C}$, we have a well defined homomorphism
$J_C[3]/\lan \xi\ran \to \wt{P}$, where $J_C[3]:=\ker([3]_{J_C})$. Let us consider the induced homomorphism of finite groups
$$J_C[3]/\lan \xi\ran \to \wt{P}/P.$$
Since $C$ has genus $2$, we have $|J_C[3]/\lan \xi\ran|=27$, whereas $|\wt{P}/P|=3$. Hence, there exists a nonzero element $\eta+\lan \xi\ran$ in the kernel,
where $\eta\in J_C[3]\setminus \lan \xi\ran$. But this means that $\pi^*\eta\in P$, as claimed.

\noindent
{\bf Step 2. Choosing a point}.
We claim that for a generic point $p\in C$ one has 
$$H^*(\xi(p))=H^*(\xi^{-1}(p))=H^*(\eta(p))=H^*(\eta\ot \xi(p))=H^*(\eta\ot \xi^{-1}(p)).$$
Indeed, for any nontrivial line bundle $M$ of degree $0$ one has $h^0(\om_C\ot M)=1$, so there exists a unique effective
divisor $D_M$ such that $\om_C\ot M\simeq \OO_C(D_M)$. Now we just need to choose $p$ not in the support of $D_{\xi}$, $D_{\xi^{-1}}$, $D_{\eta\xi}$ and $D_{\eta\xi^{-1}}$.
It follows that the line bundle $\pi^*\eta\in P$ satisfies
\begin{equation}\label{pi-eta-H0-vanishing}
H^0(\wt{C},\pi^*\eta(\pi^{-1}(p)))=H^0(C,\eta(p)\oplus \eta\xi(p)\oplus \eta\xi^{-1}(p))=0.
\end{equation}

\noindent
{\bf Step 3. Choosing a line bundle in $P$}.
We claim that there exists a line bundle $L$ of degree $1$ on $\wt{C}$ such that
$\Nm_{\wt{C}/C}(L)\simeq \OO_C(p)$, $H^0(\wt{C},L)=0$, and 
$$H^*(\wt{C},\tau_*L\ot L^{-1}(\pi^{-1}(p)))=H^*(\wt{C},\tau^{-1}_*L\ot L^{-1}(\pi^{-1}(p)))=0.$$ 

Indeed, let us look for $L$ in the form $L=M(q)$, where $p=\pi(q)$ and $M\in P$. We claim that our conditions are satisfied for generic $M\in P$.
Indeed, let us denote by $\Th_p\sub J_{\wt{C}}$ the theta-divisor associated with $\pi^{-1}(p)$.
We need to check that the following conditions hold for generic $M\in P$:
\begin{enumerate}
\item $H^0(\wt{C},M(q))=0$;
\item $\phi^\pm (M)\not\in \Th_p$, where $\phi^\pm=t_{\tau^{\pm 1}(q)-q}\circ (\tau^{\pm 1}_*-\id)$, where $t_{\tau^{\pm 1}(q)-q}$ is the translation by $\OO(\tau^{\pm 1}(q)-q)\in P$.
\end{enumerate} 
%$$H^*(\wt{C},\tau_*L_0\ot L_0^{-1}(2\tau(q)+\tau^{-1}(q)))=H^*(\wt{C},\tau^{-1}_*L_0\ot L_0^{-1}(2\tau^{-1}(q)+\tau(q)))=0.$$
For (1), we observe that this condition holds for $M=\pi^*\eta$ by \eqref{pi-eta-H0-vanishing}, since 
$H^0(\pi^*\eta(q))\sub H^0(\pi^*\eta(\pi^{-1}(p)))$.
For (2), we observe that the maps $\phi^\pm:P\to P$ are surjective, hence, it is enough to prove the existence of $M'\in P$ such that $M'\not\in \Th_p$. By \eqref{pi-eta-H0-vanishing},
we can take $M'=\pi^*\eta$.

\noindent
{\bf Step 4. Checking the properties of the vector bundle}.
Now we set $V=\pi_*(L)$. We have $\det(V)\simeq \Nm_{\wt{C}/C}(L)\simeq \OO_C(p)$, so
$$H^0(\det(V)^{-1})=H^0(\OO_C(-p))=0,$$ 
while $H^1(\det(V))=H^1(\OO_C(p))\neq 0$. We also have
$$H^0(C,V)\simeq H^0(\wt{C},L)=0,$$
$$H^0(C,\om_C\ot V)\simeq H^0(\wt{C},\pi^*\om_C\ot L)\simeq H^0(\wt{C},\om_{\wt{C}}\ot L)\neq 0,$$
since $\deg(L)=1$.
Finally, since $\pi$ is unramified, we have an isomorphism
$V^\vee\simeq \pi_*(L^{-1})$, so
$$V^\vee\ot V\simeq \pi_*(\pi^*\pi_*(L)\ot L^{-1}).$$
Furthermore, $\pi^*\pi_*(L)\simeq L\oplus \tau^*L\oplus (\tau^{-1})^*L$. Hence,
$$V^\vee\ot V\simeq \pi_*\OO\oplus \pi_*(\tau^*L\ot L^{-1})\oplus \pi_*((\tau^{-1})^*L\ot L^{-1}).$$
We need to check that 
$$h^1(V^\vee\ot V\ot \det(V))=h^1(V^\vee\ot V(p))=1.$$
Since $\pi_*\OO\ot \OO_C(p)=\OO_C(p)\oplus \xi(p)\oplus \xi^{-1}(p)$, this follows from the vanishing
$$H^1(C,\pi_*((\tau^{\pm 1})^*L\ot L^{-1})(p))=H^1(\wt{C},(\tau^{\pm 1})^*L\ot L^{-1}(\pi^{-1}(p)))=0.$$
\end{proof}

\begin{remark} It is instructive to observe that the construction of Theorem \ref{nondeg-E2-thm} breaks in characteristic $2$ 
(since Lemma \ref{exterior-power-lem} needs characteristic $\neq 2$). Indeed, by Proposition \ref{char-2-prop}, in characteristic $2$ we
have $E_2^{01}(X)_-=0$.
\end{remark}

\end{document}